\documentclass[reqno]{amsart}
\usepackage{amssymb,amsmath,enumerate,mathrsfs}

\numberwithin{equation}{section}

\newtheorem{theorem}[equation]{Theorem}
\newtheorem{proposition}[equation]{Proposition}
\newtheorem{lemma}[equation]{Lemma}
\newtheorem{corollary}[equation]{Corollary}

\theoremstyle{definition}

\def\C{\mathbb C}
\def\E{\mathscr E}

\def\L{\mathscr L}
\def\N{\mathbb N}
\def\R{\mathbb R}
\def\Z{\mathbb Z}
\def\Dom{\mathcal D}
\def\opm{\textup{op}_{\textit{M}}}
\def\open#1{\smash[t]{\overset{{}_{\circ}}{#1}{}}}

\DeclareMathOperator{\Diff}{Diff}
\DeclareMathOperator{\End}{End}
\DeclareMathOperator{\ind}{ind}
\DeclareMathOperator{\spec}{spec}
\DeclareMathOperator{\Span}{span}
\DeclareMathOperator{\sig}{sgn}
\DeclareMathOperator{\supp}{supp}
\DeclareMathOperator{\SF}{SF}
\DeclareMathOperator{\sym}{ \sigma\!\!\!\sigma}

\begin{document}
\title[Cobordism invariance of the index revisited]{Cobordism invariance of the index for realizations of elliptic operators revisited}

\author{Thomas Krainer}
\address{Penn State Altoona\\ 3000 Ivyside Park \\ Altoona, PA 16601-3760}
\email{krainer@psu.edu}

\begin{abstract}
We revisit an argument due to Lesch \cite{Le93,Le97} for proving the cobordism invariance of the index of Dirac operators on even-dimensional closed manifolds and combine this with recent work by the author \cite{KrainerSymmetricOps} to show vanishing results for the spectral flow for families of selfadjoint Fredholm realizations of elliptic operators in case the family is induced on the boundary by an elliptic operator on a compact space. This work is motivated by studying the behavior of the index of realizations of elliptic operators under cobordisms of statified manifolds.
 \end{abstract}


\maketitle


\section{Introduction}

\noindent
One of the original proofs of the Atiyah-Singer Index Theorem is based on showing that the index of Dirac type operators is invariant under cobordisms, see Palais \cite{Palais}. This proof is analytic in nature and rooted in the classical theory of elliptic boundary value problems. Other proof strategies for the index theorem such as the heat equation proof have generally been favored because these proofs require less sophisticated analytic techniques than the original cobordism proof.

Higson \cite{Higson} gave a proof of the cobordism invariance of the index by attaching an infinite half-cylinder to the boundary and extending the operator from the manifold with boundary to the manifold with cylindrical end. The Dirac type operator on the resulting odd-dimensional complete manifold is essentially selfadjoint, and the analytic arguments involved in Higson's proof are considerably simpler compared to the original proof. Lesch \cite{Le93}, on the other hand, gave a proof by attaching a (generalized) cone to the boundary and extended the operator from the manifold with boundary to a cone operator; while conic manifolds are incomplete and thus dealing with domains of realizations of the resulting conic Dirac type operator is needed, Lesch's approach is still much simpler from a functional analytic point of view than the original proof because the maximal and minimal domains of $L^2$-based realizations in the conic case differ only by a finite-dimensional space -- the price to pay is the more intricate analysis to deal with the singularity which at this juncture has been introduced artificially. Several other analytic proofs of the cobordism invariance of the index \cite{Braverman,Nicolaescu}, a $K$-theory proof \cite{Carvalho}, and generalizations \cite{BravermanShi,Hilsum,Moroianu,Wulff} have since been found.

This note is motivated by recent advances in elliptic theory on stratified manifolds with incomplete iterated wedge metrics \cite{AlbinLeichtnamMazzeoPiazzaWitt,AlbinLeichtnamMazzeoPiazzaHodge,HartmannLeschVertmanDomain,HartmannLeschVertmanAsymptotics,NazaikinskiSavinSchulzeSternin,SchulzeFields} and gives an application of the spectral flow formula for indicial operators obtained in our recent paper \cite{KrainerSymmetricOps}. Stratified cobordisms and the cobordism invariance of the index for the signature operator have been considered in \cite{AlbinLeichtnamMazzeoPiazzaWitt,AlbinLeichtnamMazzeoPiazzaHodge}, where especially in \cite{AlbinLeichtnamMazzeoPiazzaHodge} the operator is no longer essentially selfadjoint and suitable boundary conditions associated with the singular strata are considered; stratified cobordism and the invariance of the index are used in an essential way to establish the properties of the signature of a Cheeger space considered in that paper.

From our point of view Lesch's proof \cite{Le93,Le97} of the cobordism invariance of the index is very natural in the context of elliptic theory on stratified manifolds because, unlike in the classical smooth case, singular analysis and dealing with boundary conditions associated with singular strata already are essential features of the investigations here.

In this note we will revisit and extend Lesch's proof from the Dirac case to more general operators of any order, and what amounts to the vanishing of the index in the Dirac case (for null-cobordisms) will accordingly generalize to the vanishing of the spectral flow for indicial families. Our recent paper \cite{KrainerSymmetricOps} on indicial operators, which are abstract functional analytic model operators associated to generalized conical singularities, is the basis for this. We will only be concerned with null-cobordisms and proving vanishing results here; more general notions of cobordisms and cobordism invariance follow upon reduction to this case. Without detailing the precise assumptions, the argument proceeds as follows:

Let $(M,g)$ be a Riemannian manifold, and let $U = U(Y) \subset M$ be an open subset that is isometric to $(0,\varepsilon) \times Y$ with product metric $dx^2 + g_Y$ for some $\varepsilon > 0$, where $(Y,g_Y)$ is another Riemannian manifold. The reader ought to think of both $M$ and $Y$ as the open interior of compact stratified manifolds $\overline{M}$ and $\overline{Y}$ equipped with incomplete iterated wedge metrics, where $\overline{Y}$ is a boundary hypersurface of $\overline{M}$, and $U(Y)$ is a collar neighborhood. Let $\E \to M$ be a Hermitian vector bundle such that $\E\big|_{U(Y)} \cong \pi_Y^*E$ isometrically, where $E \to Y$ is a Hermitian vector bundle, and $\pi_Y : (0,\varepsilon) \times Y \to Y$ is the canonical projection. Let
$$
A : C_c^{\infty}(M;\E) \to C_c^{\infty}(M;\E)
$$
be an elliptic differential operator of order $\mu \geq 1$ that is symmetric with respect to the inner product induced by the Riemannian and Hermitian metrics, and suppose that $A$ is in $U(Y)$ of the form
$$
A \cong A_{\wedge} = x^{-1}\sum\limits_{j=0}^{\mu}a_j(y,D_y)(xD_x)^j : C_c^{\infty}((0,\varepsilon)\times Y;\pi_Y^*E) \to C_c^{\infty}((0,\varepsilon)\times Y;\pi_Y^*E),
$$
where $a_j(y,D_y) \in \Diff^{\mu-j}(Y;E)$. Let
$$
p(\sigma) = \sum\limits_{j=0}^{\mu}a_j(y,D_y)\sigma^j : C_c^{\infty}(Y;E) \to C_c^{\infty}(Y;E), \; \sigma \in \C,
$$
be the indicial family. Now suppose that
\begin{equation}\label{AminOpIntro}
A_{\min} : \Dom_{\min}(A) \subset L^2(M;\E) \to L^2(M;\E)
\end{equation}
is some closed symmetric extension of $A : C_c^{\infty}(M;\E) \subset L^2(M;\E) \to L^2(M;\E)$, and let $A_{\max} : \Dom_{\max}(A) \subset L^2(M;\E) \to L^2(M;\E)$ be the adjoint -- we point out here that $A_{\min}$ is not necessarily the minimal extension of $A$ from $C_c^{\infty}(M;\E)$, and therefore $A_{\max}$ is not the largest $L^2$-based closed extension either, i.e. we only have
\begin{align*}
\Dom_{\min}(A) &\supset \{u \in L^2(M;\E);\; \exists u_k \in C_c^{\infty}(M;\E), u_k \to u \textup{ in } L^2(M;\E), \\
&\qquad\qquad \textup{ and } Au_k \subset L^2(M;\E) \textup{ Cauchy}\}, \\
\Dom_{\max}(A) &\subset \{u \in L^2(M;\E);\; \exists v \in L^2(M;\E) : \\
&\qquad\qquad \langle A\phi,u \rangle_{L^2(M;\E)} = \langle \phi,v \rangle_{L^2(M;\E)} \;\forall \phi \in C_c^{\infty}(M;\E)\},
\end{align*}
and these inclusions are generally proper. The reader ought to think of the operator $A$ as an elliptic iterated incomplete wedge operator on $\overline{M}$, and the domain $\Dom_{\min}(A)$ as determined by previously chosen boundary conditions for $A$ associated with singular strata of $\overline{M}$ away from the boundary hypersurface $\overline{Y} \subset \overline{M}$.

One of the main points now is that under suitable localization and compatibility assumptions these extensions of $A$ should localize to $U(Y)$ and be fully captured by the extensions of the indicial operator
\begin{equation}\label{Awedgeintro}
A_{\wedge} : C_c^{\infty}(\R_+;E_1) \subset L^2(\R_+\times Y;\pi_Y^*E) \to L^2(\R_+\times Y;\pi_Y^*E).
\end{equation}
Here
$$
H^{\mu}_{\textup{comp}}(Y;E) \subset E_1 \subset H^{\mu}_{\textup{loc}}(Y;E)
$$
is the common domain for the indicial family $p(\sigma) : E_1 \subset  E_0 \to E_0$, $\sigma \in \C$, where $E_0 = L^2(Y;E)$, giving rise to a holomorphic family of unbounded Fredholm operators that are selfadjoint for $\sigma \in \R$. The reader ought to think of $E_1$ as determined by certain lateral boundary conditions associated with the singular strata of $\overline{Y}$, obtained via restriction to $U(Y)$ by the previously determined boundary conditions on $\overline{M}$ for $A$ that gave rise to $\Dom_{\min}(A)$; the localization and compatibility assumptions are such that the boundary conditions previously chosen for $A$ on $\overline{M}$ should be selfadjoint away from the boundary hypersurface $\overline{Y}$. The upshot of all of this is that we obtain a unitary equivalence
$$
\bigl(\Dom_{\max}(A)/\Dom_{\min}(A),[\cdot,\cdot]_A\bigr) \cong \bigl(\Dom_{\max}(A_{\wedge})/\Dom_{\min}(A_{\wedge}),[\cdot,\cdot]_{A_{\wedge}}\bigr)
$$
of finite-dimensional indefinite inner product spaces by passing to representatives supported in $U(Y) \cong (0,\varepsilon)\times Y$, thus allowing transitioning between $M$ and $\R_+\times Y$; here
\begin{gather*}
[\cdot,\cdot]_A : \Dom_{\max}(A) \times \Dom_{\max}(A) \to \C, \\
[u,v]_A = \frac{1}{i}\Bigl[\langle A_{\max}u,v \rangle_{L^2} - \langle u,A_{\max}v \rangle_{L^2}\Bigr]
\end{gather*}
is the adjoint pairing, and likewise for $[\cdot,\cdot]_{A_{\wedge}}$, while $\Dom_{\min}(A_{\wedge})$ is the domain of the closure $A_{\wedge,\min}$ of \eqref{Awedgeintro}, and $\Dom_{\max}(A_{\wedge})$ is the domain of the adjoint $A_{\wedge,\max} = A_{\wedge,\min}^*$. In particular, we have
$$
\sig\bigl(\Dom_{\max}(A)/\Dom_{\min}(A),[\cdot,\cdot]_A\bigr) = \sig\bigl(\Dom_{\max}(A_{\wedge})/\Dom_{\min}(A_{\wedge}),[\cdot,\cdot]_{A_{\wedge}}\bigr)
$$
for the signatures of these spaces. On the one hand, using the spectral flow formula from \cite{KrainerSymmetricOps}, we have
\begin{gather*}
\sig\bigl(\Dom_{\max}(A_{\wedge})/\Dom_{\min}(A_{\wedge}),[\cdot,\cdot]_{A_{\wedge}}\bigr) = 
\SF[\,p(\sigma) : E_1 \subset  E_0 \to  E_0,\; -\infty < \sigma < \infty\,],
\end{gather*}
while on the other hand $\sig\bigl(\Dom_{\max}(A)/\Dom_{\min}(A),[\cdot,\cdot]_A\bigr) = 0$ if \eqref{AminOpIntro} is Fredholm or the embedding $\Dom_{\max}(A) \hookrightarrow L^2(M;\E)$ is compact, which combined leads to the desired conclusion that
$$
\SF[\,p(\sigma) : E_1 \subset  E_0 \to  E_0,\; -\infty < \sigma < \infty\,] = 0.
$$
In the Dirac case the spectral flow of the indicial family is easily seen to equal the Fredholm index of the operator $D : \Dom(D) \subset L^2(Y;E_-) \to L^2(Y;E_+)$ on the even-dimensional boundary $Y$, thus recovering cobordism invariance of the index in this context.

The structure of this paper is as follows: In Section~\ref{sec-ExtensionTheory} we briefly review what is needed from extension theory of symmetric operators, in particular the criteria that ensure that $\bigl(\Dom_{\max}(A)/\Dom_{\min}(A),[\cdot,\cdot]_A\bigr)$ is finite-dimensional with signature zero. In Section~\ref{sec-indicialoperators} we review results from our paper \cite{KrainerSymmetricOps} on indicial operators in the form in which they are needed here; we also address in this section how indicial operators of first order that model the Dirac case fit into this framework in order to obtain the desired conclusions about the cobordism invariance of the index when specializing to such operators. In Section~\ref{sec-Theorem} we fill in the details of the outline above and prove the null-cobordism theorem (Theorem~\ref{CobordismInvarianceProductType}). Finally, in Appendix~\ref{sec-closedmanifolds}, we discuss the null-cobordism theorem for smooth manifolds; assumptions appear much weaker here on the geometry and the participating objects because the analytic tools available in this case are rich enough to create the preconditions needed to apply the null-cobordism theorem rather than having to assume them from the outset. With the ongoing further development of singular analysis on stratified manifolds we anticipate similar reductions and simplifications for such cases in the future as well.


\section{Preliminaries from extension theory}\label{sec-ExtensionTheory}

\noindent
Let $H$ be a separable complex Hilbert space, and suppose $A_{\min} : \Dom_{\min} \subset H \to H$ is closed, densely defined, and symmetric. Let $A_{\max}:= A_{\min}^* : \Dom_{\max} \subset H \to H$ be the adjoint. We equip $\Dom_{\max}$ with the graph inner product
$$
\langle u,v \rangle_{A_{\max}} = \langle u,v \rangle + \langle A_{\max}u,A_{\max}v \rangle
$$
and associated graph norm. Then $\Dom_{\min} \subset \bigl(\Dom_{\max},\|\cdot\|_{A_{\max}}\bigr)$ is a closed subspace, and
$$
\Dom_{\max} = \Dom_{\min} \oplus \ker(A_{\max} + i) \oplus \ker(A_{\max} - i)
$$
by von Neumann's formulas. The dimensions
$$
n_{\pm} = \dim\ker(A_{\max} \mp  \lambda i) \in \N_0\cup\{\infty\}, \quad \lambda > 0,
$$
are the deficiency indices of the operator $A_{\min}$ and independent of $\lambda >0$. The operators
$$
A_{\min} \pm i \lambda : \Dom_{\min} \subset H \to H, \quad \lambda > 0,
$$
are injective and have closed range, and we have $n_{\pm} < \infty$ if and only if $A_{\min} \pm i\lambda$ is Fredholm, in which case $n_{\pm} = -\ind(A_{\min} \pm i\lambda)$. The adjoint pairing
\begin{gather*}
[\cdot,\cdot]_A : \Dom_{\max} \times \Dom_{\max} \to \C, \\
[u,v]_A = \frac{1}{i}\Bigl[\langle A_{\max}u,v \rangle - \langle u,A_{\max}v \rangle\Bigr]
\end{gather*}
descends to a nondegenerate Hermitian sesquilinear form (indefinite inner product)
$$
[\cdot,\cdot] : \Dom_{\max}/\Dom_{\min} \times \Dom_{\max}/\Dom_{\min} \to \C.
$$
If $\dim\Dom_{\max}/\Dom_{\min} < \infty$, i.e. if $A_{\min}$ has finite deficiency indices, the signature of the adjoint pairing is given by
$$
\sig\bigl(\Dom_{\max}/\Dom_{\min},[\cdot,\cdot]\bigr) = n_+ - n_-.
$$
The following criteria are standard and useful for verification that $n_+ = n_- < \infty$.

\begin{proposition}
Suppose $A_{\min} : \Dom_{\min} \subset H \to H$ is Fredholm. Then $A_{\min}$ has finite and equal deficiency indices, and therefore
$$
\sig\bigl(\Dom_{\max}/\Dom_{\min},[\cdot,\cdot]\bigr) = 0.
$$
\end{proposition}
\begin{proof}
Because $A_{\min} : \Dom_{\min} \subset H \to H$ is Fredholm there exists $\varepsilon > 0$ such that $A_{\min} + i\lambda : \Dom_{\min} \subset H \to H$ is Fredholm for $-\varepsilon < \lambda < \varepsilon$, and consequently both $n_{\pm} < \infty$ and $A_{\min} + i\lambda$ is Fredholm for all $\lambda \in \R$. Now
$$
\R \ni \lambda \mapsto A_{\min} + i\lambda : \Dom_{\min} \subset H \to H
$$
is a continuous Fredholm function and therefore has constant index. Thus
$$
n_+ = -\ind(A_{\min} + i) = -\ind(A_{\min} - i) = n_-.
$$
\end{proof}

\begin{proposition}\label{ElementaryDeficiency2}
If the embedding $\bigl(\Dom_{\max},\|\cdot\|_{A_{\max}}\bigr) \hookrightarrow H$ is compact then $A_{\min}$ has finite and equal deficiency indices.
\end{proposition}
\begin{proof}
The norms $\|\cdot\|_{A_{\max}}$ and $\|\cdot\|_H$ are equivalent on $\ker(A_{\max} \pm i)$, and the identity map $\bigl(\ker(A_{\max} \pm i),\|\cdot\|_{A_{\max}}\bigr) \to \bigl(\ker(A_{\max} \pm i),\|\cdot\|_{H}\bigr)$ is compact by assumption. Thus $\dim\ker(A_{\max}\pm i) < \infty$. Now $A_{\min} \pm i : \Dom_{\min} \subset H \to H$ are both Fredholm, and because $\Dom_{\min} \hookrightarrow H$ is compact we have $\ind(A_{\min} - i) = \ind(A_{\min} + i)$. The proposition is proved.
\end{proof}


\section{Indicial operators}\label{sec-indicialoperators}

\noindent
We consider indicial operators of the form
\begin{equation}\label{Aindicial}
A_{\wedge} = x^{-1}\sum\limits_{j=0}^{\mu} a_j(xD_x)^j : C_c^{\infty}(\R_+;E_1) \subset L^2(\R_+;E_0) \to L^2(\R_+;E_0),
\end{equation}
where $\mu \in \N$ and $E_0$ and $E_1$ are separable complex Hilbert spaces such that $E_1 \hookrightarrow E_0$ is continuous and dense, and the operators $a_j : E_1 \to E_0$ are continuous for $j=0,\ldots,\mu$. Let
\begin{equation}\label{Aindicialfamily}
p(\sigma) = \sum\limits_{j=0}^{\mu} a_j \sigma^j : E_1 \to E_0, \quad \sigma \in \C
\end{equation}
be the indicial family associated with $A_{\wedge}$. We make the following assumptions:
\begin{enumerate}[(i)]
\item $p(\sigma) : E_1 \subset E_0 \to E_0$ is closed, densely defined, and Fredholm for $\sigma \in \C$, and the map $\C \ni \sigma \mapsto p(\sigma) \in \L(E_1,E_0)$ is holomorphic.
\item We have $p(\overline{\sigma})^* = p(\sigma) : E_1 \subset E_0 \to E_0$ as unbounded operators in $E_0$.
\item For $(\lambda,\sigma) \in \R^2$ and $|\lambda,\sigma| \geq R \gg 0$ sufficiently large $p(\sigma) + i\lambda : E_1 \to E_0$ is invertible with
$$
\sup\limits_{|\lambda,\sigma| \geq R}\Bigl\{(1 + \lambda^2 + \sigma^{2\mu})^{\frac{1}{2}}\bigl\|\bigl(p(\sigma) + i\lambda\bigr)^{-1}\bigr\|_{\L(E_0)} + \bigl\|\bigl(p(\sigma) + i\lambda\bigr)^{-1}\bigr\|_{\L(E_0,E_1)}\Bigr\} < \infty,
$$
and for every $k \in \{1,\ldots,\mu\}$ we have
$$
\sup\limits_{|\lambda,\sigma| \geq R}(1 + \lambda^2 + \sigma^{2\mu})^{\frac{k}{2\mu}}
\bigl\|\bigl[\partial_{\sigma}^kp(\sigma)\bigr]\bigl(p(\sigma) + i\lambda\bigr)^{-1}\bigr\|_{\L(E_0)} < \infty.
$$
\end{enumerate}
In \cite{KrainerSymmetricOps} we systematically studied operators of the kind \eqref{Aindicial} under such assumptions. We summarize some of the findings below:
\begin{enumerate}
\item The operator \eqref{Aindicial} is symmetric and densely defined in $L^2({\mathbb R}_+;E_0)$. Let $A_{\wedge,\min}$ be its closure, and $A_{\wedge,\max} = A_{\wedge,\min}^*$ be the adjoint. Then
$$
\dim \Dom_{\max}(A_{\wedge})/\Dom_{\min}(A_{\wedge}) < \infty,
$$
i.e., $A_{\wedge}$ has finite deficiency indices.
\item The boundary spectrum
$$
\spec_b(p) = \{\sigma \in \C;\; \textup{$p(\sigma) : E_1 \to E_0$ is not invertible}\} \subset \C
$$
is discrete, and every strip $|\Im(\sigma)| \leq K$, $K > 0$, contains only finitely many elements of $\spec_b(p)$. The elements of the boundary spectrum are generally referred to as indicial roots.
\item Fix an arbitrary cut-off function $\omega \in C_c^{\infty}(\overline{\R}_+)$ with $\omega \equiv 1$ near $x = 0$. For each indicial root $\sigma_0 \in \spec_b(p)$ let
\begin{equation}\label{Esigma0pdef}
\begin{aligned}
\E_{\sigma_0}(p) = \Bigl\{u = \omega\sum\limits_{j=0}^k e_j&\log^j(x)x^{i\sigma_0};\; k \in \N_0 \textup{ and } e_j \in E_1, \\
&\textup{and } p(\sigma)(Mu)(\sigma) \textup{ is holomorphic at $\sigma=\sigma_0$}\Bigr\},
\end{aligned}
\end{equation}
where
$$
\bigr(M u\bigl)(\sigma) = \int_0^{\infty}x^{-i\sigma}u(x)\,\frac{dx}{x}
$$
is the Mellin transform of $u$. This space is finite-dimensional for every $\sigma_0$, and we have
\begin{equation}\label{DwedgemaxE}
\Dom_{\max}(A_{\wedge}) = \Dom_{\min}(A_{\wedge}) \oplus \bigoplus\limits_{\substack{\sigma_0 \in \spec_b(p) \\ -\frac{1}{2} < \Im(\sigma_0) < \frac{1}{2}}}\E_{\sigma_0}(p).
\end{equation}
\item We have
$$
x^{\frac{1}{2}}{\mathscr H}(\R_+;E_1)\cap L^2(\R_+;E_0) \hookrightarrow \Dom_{\min}(A_{\wedge}),
$$
and $\Dom_{\min}(A_{\wedge}) = x^{\frac{1}{2}}{\mathscr H}(\R_+;E_1)\cap L^2(\R_+;E_0)$ if and only if $p(\sigma) : E_1 \to E_0$ is invertible for all $\Im(\sigma)=-\frac{1}{2}$.

The space ${\mathscr H}(\R_+;E_1)$ is the completion of $C_c^{\infty}(\R_+;E_1)$ with respect to the norm
$$
\|u\|_{{\mathscr H}}^2 = \int_{\mathbb R} \|p(\sigma+i\gamma_0)(Mu)(\sigma)\|_{E_0}^2\,d\sigma,
$$
where $\gamma_0 \in \R$ is arbitrary such that $p(\sigma+i\gamma_0) : E_1 \to E_0$ is invertible for all $\sigma \in \R$. We have
$$
{\mathscr H}(\R_+;E_1) \hookrightarrow H^{\mu}_b(\R_+;E_0)\cap L^2_b(\R_+;E_1),
$$
and in typical situations these spaces are equal; this is the case, for instance, if
\begin{equation}\label{HIntersection}
\sup\limits_{\sigma \in \R}\|p(\sigma+i\gamma_0)(\langle \sigma \rangle^{\mu} + i\Lambda)^{-1}\|_{\L(E_0)} < \infty,
\end{equation}
where $\Lambda : E_1 \subset E_0 \to E_0$ is selfadjoint (e.g. for $\Lambda = p(0)$).
\item While not discussed in \cite{KrainerSymmetricOps} it is not hard to see that, under the added assumption that the embedding $E_1 \hookrightarrow E_0$ is compact, multiplication by a cut-off function $\omega \in C_c^{\infty}(\overline{\R}_+)$ with $\omega \equiv 1$ near $x =0$ induces a compact operator $\omega : x^{\alpha}{\mathscr H}(\R_+;E_1) \to L^2_b(\R_+;E_0)$ for every $\alpha > 0$\footnote{The function $a_0(x,\sigma) = x^{\alpha}\omega(x) p(\sigma +i\gamma_0)^{-1}$ is a Mellin symbol taking values in the compact operators $E_0 \to E_0$, and we have $\sup\{\langle \log(x) \rangle^j\langle \sigma \rangle^{\mu+k}\|(xD_x)^l\partial_{\sigma}^ka_0(x,\sigma)\|_{\L(E_0)};\; (x,\sigma) \in \R_+\times\R\} < \infty$ for all $j,k,l \in \N_0$. Thus the Mellin pseudodifferential operator $\opm(a_0) : L^2_b(\R_+;E_0) \to L^2_b(\R_+;E_0)$ is compact, which implies compactness of the multiplication operator $\omega : x^{\alpha}{\mathscr H}(\R_+;E_1) \to L^2_b(\R_+;E_0)$ as asserted.}.

Consequently, if additionally $p(\sigma) : E_1 \to E_0$ is invertible for all $\Im(\sigma) = -\frac{1}{2}$, we obtain a compact map $\omega : \Dom_{\max}(A_{\wedge}) \to L^2(\R_+;E_0)$, and a bounded map $1-\omega : \Dom_{\max}(A_{\wedge}) \to \Dom_{\min}(A_{\wedge})$. The latter is based on the identity $\Dom_{\min}(A_{\wedge}) = x^{\frac{1}{2}}{\mathscr H}(\R_+;E_1)\cap L^2(\R_+;E_0)$ and localization properties of the space ${\mathscr H}(\R_+;E_1)$ (see \cite[Proposition~7.6]{KrainerSymmetricOps}).

\item The adjoint pairing
\begin{gather*}
[\cdot,\cdot]_{A_{\wedge}} : \Dom_{\max}(A_{\wedge}) \times \Dom_{\max}(A_{\wedge}) \to \C, \\
[u,v]_{A_{\wedge}} = \frac{1}{i}\Bigl[\langle A_{\wedge,\max}u,v \rangle_{L^2(\R_+;E_0)} - \langle u,A_{\wedge,\max}v \rangle_{L^2(\R_+;E_0)}\Bigr]
\end{gather*}
induces a nondegenerate Hermitian sesquilinear form
$$
[\cdot,\cdot] : \Dom_{\max}(A_{\wedge})/\Dom_{\min}(A_{\wedge}) \times \Dom_{\max}(A_{\wedge})/\Dom_{\min}(A_{\wedge}) \to \C,
$$
and its signature is given by the spectral flow of the indicial family \eqref{Aindicialfamily} along the real line:
\begin{equation}\label{SignatureSpectralFlow1}
\sig\bigl(\Dom_{\max}(A_{\wedge})/\Dom_{\min}(A_{\wedge}),[\cdot,\cdot]\bigr) = \SF[\,p(\sigma) : E_1 \subset E_0 \to E_0,\; -\infty < \sigma < \infty\,].
\end{equation}
Note that $p(\sigma) : E_1 \to E_0$ is invertible for $|\sigma| \geq T \gg 0$ large enough, $\sigma \in \R$, and the spectral flow in \eqref{SignatureSpectralFlow1} then refers to $p(\sigma)$ on the interval $-T \leq \sigma \leq T$. Only crossings of real indicial roots contribute terms to the spectral flow.
\end{enumerate}

The focus in this paper is on the signature of the adjoint pairing, and by \eqref{SignatureSpectralFlow1} only real indicial roots are relevant. In order to obtain simple expressions for the minimal domain and the maximal domain \eqref{DwedgemaxE} of $A_{\wedge}$ it is sometimes convenient to introduce a scaling parameter $t > 0$ to remove any small non-real indicial roots from the strip $|\Im(\sigma)| \leq \frac{1}{2}$. This leads to
$$
A_{\wedge,t} = x^{-1}\sum\limits_{j=0}^{\mu} a_jt^j(xD_x)^j : C_c^{\infty}(\R_+;E_1) \subset L^2(\R_+;E_0) \to L^2(\R_+;E_0)
$$
with indicial family
$$
p_t(\sigma) = p(t\sigma) : E_1 \subset E_0 \to E_0, \quad \sigma \in \C,
$$
and the standing assumptions on $p(\sigma)$ imply that the analogous properties are also satisfied for $p_t(\sigma)$, and all estimates are locally uniform with respect to $t > 0$. In particular, the spectral flow
$$
\SF[\,p_t(\sigma) : E_1 \subset E_0 \to E_0,\; -\infty < \sigma < \infty\,]
$$
is independent of $t > 0$ by homotopy invariance, and thus
$$
\sig\bigl(\Dom_{\max}(A_{\wedge,t})/\Dom_{\min}(A_{\wedge,t}),[\cdot,\cdot]\bigr)
$$
is independent of $t > 0$. For $0 < t \leq t_0$ small enough, $p_t(\sigma) : E_1 \subset E_0 \to E_0$ is invertible for all $0 < |\Im(\sigma)| \leq \frac{1}{2}$. We then have
$$
\Dom_{\min}(A_{\wedge,t}) = x^{\frac{1}{2}}{\mathscr H}(\R_+;E_1)\cap L^2(\R_+;E_0),
$$
where the definition of ${\mathscr H}(\R_+;E_1)$ is accordingly based on $p_t(\sigma)$, and
$$
\Dom_{\max}(A_{\wedge,t}) = \Dom_{\min}(A_{\wedge,t}) \oplus \bigoplus\limits_{\sigma_0 \in \spec_b(p_t)\cap\R}\E_{\sigma_0}(p_t).
$$
If \eqref{HIntersection} holds for $p(\sigma)$ it is true for all $p_t(\sigma)$, and in this case the space
$$
{\mathscr H}(\R_+;E_1) = H^{\mu}_b(\R_+;E_0)\cap L^2_b(\R_+;E_1)
$$
is independent of $t > 0$; thus the minimal domain
$$
\Dom_{\min}(A_{\wedge}) = x^{\frac{1}{2}}H^{\mu}_b(\R_+;E_0)\cap x^{\frac{1}{2}}L^2_b(\R_+;E_1)\cap L^2(\R_+;E_0)
$$
is independent of $0 < t \leq t_0$.

\subsection*{Operators of first order}

Let $D : \Dom(D)\subset H_1 \to H_2$ be closed and densely defined, and let $D^* : \Dom(D^*) \subset H_2 \to H_1$ be the adjoint. Write
$$
E_0 = \begin{array}{c} H_1 \\ \oplus \\ H_2 \end{array} \textup{ and }
E_1 = \begin{array}{c} \Dom(D) \\ \oplus \\ \Dom(D^*) \end{array} \hookrightarrow E_0.
$$
We assume that $D$ (and therefore also $D^*$) is Fredholm, and that the embeddings for both domains $\Dom(D) \hookrightarrow H_1$ and $\Dom(D^*) \hookrightarrow H_2$ are compact. Consider then
$$
{\mathscr D}_{\wedge} = x^{-1}\biggl[\begin{bmatrix} 1 & 0 \\ 0 & -1 \end{bmatrix}(xD_x) + \begin{bmatrix} 0 & D^* \\ D & 0 \end{bmatrix}\biggr] : C_c^{\infty}(\R_+;E_1) \subset L^2(\R_+;E_0) \to L^2(\R_+;E_0)
$$
with indicial family
$$
{\mathscr D}(\sigma) = \begin{bmatrix} \sigma & D^* \\ D & -\sigma \end{bmatrix} : E_1 \subset E_0 \to E_0, \quad \sigma \in \C.
$$
Now ${\mathscr D}(\sigma)$ satisfies the assumptions previously stated for indicial families with $\mu = 1$, including \eqref{HIntersection} with $\Lambda = {\mathscr D}(0)$; see Lemma~\ref{FirstOrderIndicialAssumption} for the required estimates. Therefore the conclusions summarized above hold for ${\mathscr D}_{\wedge}$, and by Lemma~\ref{SpectralFlowIndex} we have
\begin{equation}\label{IndexEqualsSignature}
\sig\bigl(\Dom_{\max}({\mathscr D}_{\wedge})/\Dom_{\min}({\mathscr D}_{\wedge}),[\cdot,\cdot]\bigr) = \ind[D : \Dom(D) \subset H_1 \to H_2].
\end{equation}
The only real indicial root is $\sigma_0 = 0$, and after possibly introducing a sufficiently small scaling parameter $t > 0$ and replacing ${\mathscr D}_{\wedge}$ by
$$
{\mathscr D}_{\wedge,t} = x^{-1}\biggl[t\begin{bmatrix} 1 & 0 \\ 0 & -1 \end{bmatrix}(xD_x) + \begin{bmatrix} 0 & D^* \\ D & 0 \end{bmatrix}\biggr]
$$
we have
\begin{gather*}
\Dom_{\min}({\mathscr D}_{\wedge,t}) = x^{\frac{1}{2}}H^1_b(\R_+;E_0)\cap x^{\frac{1}{2}}L^2_b(\R_+;E_1)\cap L^2(\R_+;E_0), \\
\Dom_{\max}({\mathscr D}_{\wedge,t}) = \Dom_{\min}({\mathscr D}_{\wedge,t}) \oplus \E_{0}({\mathscr D}_t).
\end{gather*}
In this case $\E_{0}({\mathscr D}_t) = \E_{0}({\mathscr D})$ is also independent of $t > 0$, and we have
$$
\E_{0}({\mathscr D}) = \Bigl\{u=\omega\begin{bmatrix} k \\ k^* \end{bmatrix};\; k \in \ker(D),\; k^* \in \ker(D^*)\Bigr\}.
$$
This follows from \eqref{Esigma0pdef} in view of
\begin{align*}
{\mathscr D}(\sigma)^{-1} &= \sigma\begin{bmatrix} 1 & 0 \\ 0 & -1 \end{bmatrix}[{\mathscr D}(0)^2 + \sigma^2]^{-1} + {\mathscr D}(0)[{\mathscr D}(0)^2 + \sigma^2]^{-1} \\
&= \begin{bmatrix} \Pi_{D} & 0 \\ 0 & -\Pi_{D^*} \end{bmatrix}\frac{1}{\sigma} + \textup{holomorphic}
\end{align*}
near $\sigma = 0$, where $\Pi_{D} : H_1 \to \ker(D)$ and $\Pi_{D^*} : H_2 \to \ker(D^*)$ are the orthogonal projections onto the kernels of $D$ and $D^*$, respectively. For sufficiently small $t > 0$ a brief calculation shows that the adjoint pairing is given by
$$
\Bigl[\omega\begin{bmatrix} k_1 \\ k_1^* \end{bmatrix},\omega\begin{bmatrix} k_2 \\ k_2^* \end{bmatrix}\Bigr]_{{\mathscr D}_{\wedge,t}} = t\bigl(\langle k_1,k_2 \rangle_{H_1} - \langle k_1^*,k_2^*\rangle_{H_2}\bigr)
$$
for $k_j \in \ker(D)$ and $k_j^* \in \ker(D^*)$, $j=1,2$, which provides a direct justification for \eqref{IndexEqualsSignature} for ${\mathscr D}_{\wedge,t}$ (for small $t > 0$) that does not rely on the spectral flow.

\begin{lemma}\label{FirstOrderIndicialAssumption}
For $(\lambda,\sigma) \in \R^2$ write $z = \sigma + i\lambda \in \C$ and consider
$$
{\mathbf D}(z) = {\mathscr D}(\sigma) + i\lambda = \begin{bmatrix}
z & D^* \\
D & -\overline{z}
\end{bmatrix} :
E_1 \subset E_0 \to E_0.
$$
Then ${\mathbf D}(z)$ is invertible for all $z \in \C\setminus\{0\}$, and
$$
\sup\limits_{|z|\geq 1}\{|z|\cdot \|{\mathbf D}(z)^{-1}\|_{\L(E_0)} + \|{\mathbf D}(z)^{-1}\|_{\L(E_0,E_1)}\} < \infty.
$$
\end{lemma}
\begin{proof}
We have ${\mathbf D}(z)^* = {\mathbf D}(\overline{z})$, and
$$
{\mathbf D}(z)^*{\mathbf D}(z) = {\mathbf D}(z){\mathbf D}(z)^* = \begin{bmatrix} |z|^2 + D^*D & 0 \\ 0 & |z|^2 + DD^* \end{bmatrix} = {\mathscr D}(0)^2 + |z|^2.
$$
This operator is invertible for $z \in \C \setminus \{0\}$, and consequently ${\mathbf D}(z)$ is invertible with
$$
{\mathbf D}(z)^{-1} = {\mathbf D}(z)^*[{\mathbf D}(z){\mathbf D}(z)^*]^{-1} = [\overline{z}\Pi_1 - z\Pi_2][{\mathscr D}(0)^2 + |z|^2]^{-1} + {\mathscr D}(0)[{\mathscr D}(0)^2 + |z|^2]^{-1},
$$
where $\Pi_j : E_0 \to H_j \subset E_0$ is the orthogonal projection, $j=1,2$. In view of ${\mathscr D}(0)[\overline{z}\Pi_1 - z\Pi_2] = [\overline{z}\Pi_2 - z\Pi_1]{\mathscr D}(0)$ we have
$$
{\mathscr D}(0){\mathbf D}(z)^{-1}  = [\overline{z}\Pi_2 - z\Pi_1]{\mathscr D}(0)[{\mathscr D}(0)^2 + |z|^2]^{-1} + {\mathscr D}(0)^2[{\mathscr D}(0)^2 + |z|^2]^{-1}.
$$
The Spectral Theorem implies
$$
\sup\limits_{|z|\geq 1}\{\|{\mathscr D}(0)^2[{\mathscr D}(0)^2 + |z|^2]^{-1}\| + \|{z\mathscr D}(0)[{\mathscr D}(0)^2 + |z|^2]^{-1}\| + \|z^2[{\mathscr D}(0)^2 + |z|^2]^{-1}\|\} < \infty,
$$
where $\|\cdot\| = \|\cdot\|_{\L(E_0)}$. The lemma now follows.
\end{proof}

\begin{lemma}\label{SpectralFlowIndex}
We have
$$
\ind[D : \Dom(D) \subset H_1 \to H_2] = \SF\bigl[{\mathscr D}(\sigma) : E_1 \subset E_0 \to E_0,\; \sigma \in \R\bigr].
$$
\end{lemma}
\begin{proof}
Let ${\mathcal K} = \ker({\mathscr D}(0)) = \ker(D) \oplus \ker(D^*)$. Then
$$
{\mathscr D}(\sigma) = \begin{bmatrix} {\mathscr D}_{{\mathcal K}}(\sigma) & 0 \\ 0 & {\mathscr D}_{{\mathcal K}^{\perp}}(\sigma) \end{bmatrix} : \begin{array}{c} {\mathcal K} \\ \oplus \\ {\mathcal K}^{\perp}\cap E_1 \end{array} \to \begin{array}{c} {\mathcal K} \\ \oplus \\ {\mathcal K}^{\perp} \end{array}, \quad \sigma \in \R.
$$
Now ${\mathscr D}_{{\mathcal K}}(\sigma) : {\mathcal K} \to {\mathcal K}$, $\sigma \neq 0$, has eigenvalues $\sigma,-\sigma$ of multiplicities $\dim\ker(D)$ and $\dim\ker(D^*)$, respectively, and ${\mathscr D}_{{\mathcal K}^{\perp}}(\sigma)$ is invertible for all $\sigma \in \R$. Thus
\begin{align*}
 \ind D &= \dim\ker(D) - \dim\ker(D^*) \\
&= \SF\bigl[{\mathscr D}_{{\mathcal K}}(\sigma) : {\mathcal K}  \to {\mathcal K},\; \sigma \in \R\bigr] \\
&= \SF\bigl[{\mathscr D}(\sigma) : E_1 \subset E_0 \to E_0,\; \sigma \in \R\bigr].
\end{align*}
\end{proof}


\section{The null-cobordism theorem}\label{sec-Theorem}

\noindent
We now revisit the setting discussed in the introduction to prove the null-cobordism theorem. We make the following product type assumptions on the geometry and the operator:

Let $(M,g)$ be a Riemannian manifold, and let $U = U(Y) \subset M$ be an open subset that is isometric to $(0,\varepsilon) \times Y$ with product metric $dx^2 + g_Y$ for some $\varepsilon > 0$, where $(Y,g_Y)$ is another Riemannian manifold. Let $\E \to M$ be a Hermitian vector bundle such that $\E\big|_{U(Y)} \cong \pi_Y^*E$ isometrically, where $E \to Y$ is a Hermitian vector bundle, and $\pi_Y : (0,\varepsilon) \times Y \to Y$ is the canonical projection. Let
$$
A : C_c^{\infty}(M;\E) \to C_c^{\infty}(M;\E)
$$
be an elliptic differential operator of order $\mu \geq 1$ that is symmetric with respect to the inner product induced by the Riemannian and Hermitian metrics, and suppose that $A$ is in $U(Y)$ of the form
$$
A \cong A_{\wedge} = x^{-1}\sum\limits_{j=0}^{\mu}a_j(y,D_y)(xD_x)^j : C_c^{\infty}((0,\varepsilon)\times Y;\pi_Y^*E) \to C_c^{\infty}((0,\varepsilon)\times Y;\pi_Y^*E),
$$
where $a_j(y,D_y) \in \Diff^{\mu-j}(Y;E)$. Let
$$
p(\sigma) = \sum\limits_{j=0}^{\mu}a_j(y,D_y)\sigma^j : C_c^{\infty}(Y;E) \to C_c^{\infty}(Y;E), \; \sigma \in \C,
$$
be the indicial family. We assume that $p(\sigma) : E_1 \subset E_0 \to E_0$ satisfies the assumptions stated in Section~\ref{sec-indicialoperators} with $E_0 = L^2(Y;E)$ and some domain
$$
H^{\mu}_{\textup{comp}}(Y;E) \subset E_1 \subset H^{\mu}_{\textup{loc}}(Y;E).
$$
We also assume that the embedding $E_1 \hookrightarrow E_0$ is compact, and that $p(\sigma) : E_1 \to E_0$ is invertible for $0 < |\Im(\sigma)| \leq \frac{1}{2}$; as explained in Section~\ref{sec-indicialoperators}, the latter can generally be achieved by introducing a scaling parameter (which for geometric operators typically corresponds to scaling the metric). The closed extensions of the indicial operator
$$
A_{\wedge} : C_c^{\infty}(\R_+;E_1) \subset L^2(\R_+\times Y;\pi_Y^*E) \to L^2(\R_+\times Y;\pi_Y^*E)
$$
are then described as explained in Section~\ref{sec-indicialoperators}. Let
$$
A_{\min} : \Dom_{\min}(A) \subset L^2(M;\E) \to L^2(M;\E)
$$
be a closed symmetric extension of $A : C_c^{\infty}(M;\E) \subset L^2(M;\E) \to L^2(M;\E)$, and let $A_{\max} : \Dom_{\max}(A) \subset L^2(M;\E) \to L^2(M;\E)$ be the adjoint; as discussed in the introduction, $A_{\min}$ is generally not the minimal extension of $A$ from $C_c^{\infty}(M;\E)$, and thus $A_{\max}$ is not the largest $L^2$-based closed extension. By elliptic regularity we have
$$
H^{\mu}_{\textup{comp}}(M;\E) \subset \Dom_{\min}(A) \subset \Dom_{\max}(A) \subset H^{\mu}_{\textup{loc}}(M;\E).
$$
By a cut-off function we mean any function $\omega \in C_c^{\infty}([0,\varepsilon))$ such that $\omega \equiv 1$ near $x=0$, and we consider $\omega$ a function on $M$ supported in $U(Y)$. We make the following localization and compatibility assumptions between $A$ and $A_{\wedge}$:
\begin{itemize}
\item For every cut-off function $\omega$, multiplication by $1-\omega$ gives a continuous operator $\Dom_{\max}(A) \to \Dom_{\min}(A)$. We also assume that $1 - \omega : \Dom_{\min}(A) \to L^2(M;\E)$ is compact.
\item For every cut-off function $\omega$, multiplication by $\omega$ gives continuous operators $\Dom_{\min}(A) \to \Dom_{\min}(A_{\wedge})$ and $\Dom_{\min}(A_{\wedge}) \to \Dom_{\min}(A)$.
\end{itemize}

\noindent
To make sense of the mappings above note that
$$
M \supset U(Y) \cong (0,\varepsilon)\times Y \subset \R_+ \times Y,
$$
which allows transitioning both ways between functions on $M$ supported in $U(Y)$ and functions on $\R_+\times Y$ supported in $(0,\varepsilon)\times Y$. We will use these transitions freely in what follows.

\begin{proposition}\label{unitaryequiv}
Let $\omega \in C_c^{\infty}([0,\varepsilon))$ be any cut-off function. The map
$$
\Dom_{\max}(A)/\Dom_{\min}(A) \ni u + \Dom_{\min}(A) \longmapsto \omega u + \Dom_{\min}(A_{\wedge}) \in \Dom_{\max}(A_{\wedge})/\Dom_{\min}(A_{\wedge})
$$
is well-defined, and induces a unitary equivalence between the indefinite inner product spaces
$$
\bigl(\Dom_{\max}(A)/\Dom_{\min}(A),[\cdot,\cdot]_A\bigr) \cong \bigl(\Dom_{\max}(A_{\wedge})/\Dom_{\min}(A_{\wedge}),[\cdot,\cdot]_{A_{\wedge}}\bigr).
$$
\end{proposition}
\begin{proof}
We first prove that multiplication by $\omega$ gives a well-defined map
$$
\Dom_{\max}(A) \ni u \mapsto \omega u \in \Dom_{\max}(A_{\wedge}).
$$
Note that with $u$ also $\omega u \in \Dom_{\max}(A)$ by our localization assumption. Now pick another cut-off function $\tilde{\omega} \in C_c^{\infty}([0,\varepsilon))$ such that $\tilde{\omega} \equiv 1$ in a neighborhood of $\supp(\omega)$. Let $\phi \in \Dom_{\min}(A_{\wedge})$ be arbitrary, and write $\phi = \tilde{\omega}\phi + (1-\tilde{\omega})\phi$. Since
$$
\Dom_{\min}(A_{\wedge}) = x^{\frac{1}{2}}{\mathscr H}(\R_+;E_1)\cap L^2(\R_+;E_0)
$$
as a consequence of our assumptions we have that both $\tilde{\omega}\phi,\, (1-\tilde{\omega})\phi \in \Dom_{\min}(A_{\wedge})$, see Section~\ref{sec-indicialoperators}. We also have $\tilde{\omega}\phi \in \Dom_{\min}(A)$ by our localization and compatibility assumption with respect to the minimal domains. Using the locality of the differential operators $A_{\wedge}$ and $A$ we get
\begin{align*}
\langle A_{\wedge}\phi,\omega u \rangle &= \langle A_{\wedge}(\tilde{\omega}\phi),\omega u \rangle = \langle A(\tilde{\omega}\phi),\omega u \rangle = \langle \tilde{\omega}\phi,A_{\max}(\omega u) \rangle \\
&= \langle \phi,\tilde{\omega}A_{\max}(\omega u) \rangle = \langle \phi,A_{\max}(\omega u) \rangle.
\end{align*}
As this is valid for all $\phi \in \Dom_{\min}(A_{\wedge})$ we see that $\omega u \in \Dom_{\max}(A_{\wedge})$ with $A_{\wedge,\max}(\omega u)$ given as the restriction of $A_{\max}(\omega u)$ to $U(Y)$ and extended trivially to ${\mathbb R}_+\times Y$.
As for $u \in \Dom_{\min}(A)$ we also have $\omega u \in \Dom_{\min}(A_{\wedge})$ by assumption, we thus obtain that the map
$$
\Dom_{\max}(A)/\Dom_{\min}(A) \ni u + \Dom_{\min}(A) \longmapsto \omega u + \Dom_{\min}(A_{\wedge}) \in \Dom_{\max}(A_{\wedge})/\Dom_{\min}(A_{\wedge})
$$
is well-defined.

Conversely, multiplication by $\omega$ likewise gives a well-defined map
$$
\Dom_{\max}(A_{\wedge}) \ni u \mapsto \omega u \in \Dom_{\max}(A).
$$
Note that if $u \in \Dom_{\max}(A_{\wedge})$ then $\omega u \in \Dom_{\max}(A_{\wedge})$ and $(1-\omega)u \in \Dom_{\min}(A_{\wedge})$ by Section~\ref{sec-indicialoperators}. Now let $\tilde{\omega} \in C_c^{\infty}([0,\varepsilon))$ be such that $\tilde{\omega} \equiv 1$ in a neighborhood of $\supp(\omega)$. Let $\phi \in \Dom_{\min}(A)$ be arbitrary, and write $\phi = \tilde{\omega}\phi + (1-\tilde{\omega})\phi$; by the localization and compatibility assumptions both terms are in $\Dom_{\min}(A)$, and we also have $\tilde{\omega}\phi \in \Dom_{\min}(A_{\wedge})$.
We get
\begin{align*}
\langle A\phi,\omega u \rangle &= \langle A(\tilde{\omega}\phi),\omega u \rangle = \langle A_{\wedge}(\tilde{\omega}\phi),\omega u \rangle = \langle \tilde{\omega}\phi,A_{\wedge,\max}(\omega u) \rangle \\
&= \langle \phi,\tilde{\omega}A_{\wedge,\max}(\omega u) \rangle = \langle \phi,A_{\wedge,\max}(\omega u) \rangle.
\end{align*}
This shows that $\omega u \in \Dom_{\max}(A)$ with $A_{\max}(\omega u)$ given by $A_{\wedge,\max}(\omega u)$ in $U(Y)$ and extended trivially to $M$. We thus obtain a map
$$
\Dom_{\max}(A_{\wedge})/\Dom_{\min}(A_{\wedge}) \ni u + \Dom_{\min}(A_{\wedge}) \longmapsto \omega u + \Dom_{\min}(A) \in \Dom_{\max}(A)/\Dom_{\min}(A),
$$
and both maps are inverses of each other.

Finally, as for both $A$ and $A_{\wedge}$ each class in $\Dom_{\max}/\Dom_{\min}$ has a representative supported in $U(Y)$, and by the standing product type assumptions both adjoint pairings agree on those representatives, the proposition follows.
\end{proof}

\begin{theorem}[Null-Cobordism Theorem]\label{CobordismInvarianceProductType}
Under the stated product type, localization, and compatibility assumptions we have
$$
\SF[\,p(\sigma) : E_1 \subset E_0 \to E_0,\; -\infty < \sigma < \infty\,] = 0.
$$
If moreover
$$
p(\sigma) = \begin{bmatrix} \sigma & D^* \\ D & -\sigma \end{bmatrix} : \begin{array}{c} \Dom(D) \\ \oplus \\ \Dom(D^*) \end{array} \subset L^2\biggl(Y;\begin{array}{c} E_- \\ \oplus \\ E_+ \end{array}\biggr) \to L^2\biggl(Y;\begin{array}{c} E_- \\ \oplus \\ E_+ \end{array}\biggr)
$$
with an elliptic Fredholm operator of first order
$$
D : \Dom(D) \subset L^2(Y;E_-) \to L^2(Y;E_+),
$$
then $\ind[D : \Dom(D) \subset L^2(Y;E_-) \to L^2(Y;E_+)]  = 0$.
\end{theorem}
\begin{proof}
By Proposition~\ref{unitaryequiv} we have a unitary equivalence between the indefinite inner product spaces
$$
\bigl(\Dom_{\max}(A)/\Dom_{\min}(A),[\cdot,\cdot]_A\bigr) \cong \bigl(\Dom_{\max}(A_{\wedge})/\Dom_{\min}(A_{\wedge}),[\cdot,\cdot]_{A_{\wedge}}\bigr).
$$
Because
$$
\sig\bigl(\Dom_{\max}(A_{\wedge})/\Dom_{\min}(A_{\wedge}),[\cdot,\cdot]_{A_{\wedge}}\bigr) = \SF[\,p(\sigma) : E_1 \subset E_0 \to E_0,\; -\infty < \sigma < \infty\,]
$$
by \eqref{SignatureSpectralFlow1} it suffices to show that
$$
\sig\bigl(\Dom_{\max}(A)/\Dom_{\min}(A),[\cdot,\cdot]_A\bigr) = 0,
$$
and by Proposition~\ref{ElementaryDeficiency2} this will be the case if the embedding $\Dom_{\max}(A) \hookrightarrow L^2(M;\E)$ is compact. Because $A$ has finite deficiency indices we only need to prove that $\Dom_{\min}(A) \hookrightarrow L^2(M;\E)$ is compact. Now let $\omega,\,\tilde{\omega} \in C_c^{\infty}([0,\varepsilon))$ be cut-off functions such that $\tilde{\omega} \equiv 1$ in a neighborhood of $\supp(\omega)$. By assumption the multiplication operator
$$
1-\omega : \Dom_{\min}(A) \to L^2(M;\E)
$$
is compact, and
$$
\tilde{\omega} : \Dom_{\min}(A) \to \Dom_{\min}(A_{\wedge})
$$
is continuous. Now
$$
\Dom_{\min}(A_{\wedge}) = x^{\frac{1}{2}}{\mathscr H}(\R_+;E_1)\cap L^2(\R_+;E_0),
$$
and because $E_1 \hookrightarrow E_0$ is compact, multiplication by $\omega$ is a compact operator
$$
\omega : \Dom_{\min}(A_{\wedge}) \to L^2(\R_+;E_0),
$$
see Section~\ref{sec-indicialoperators}. Consequently, using the product type assumptions, the composition
$$
\omega = \omega\tilde{\omega} : \Dom_{\min}(A) \to L^2(M;\E)
$$
is compact, which shows that the embedding $\iota = \omega + (1-\omega) : \Dom_{\min}(A) \to L^2(M;\E)$ is compact. Finally, the vanishing of the index in the special case of operators of first order follows from \eqref{IndexEqualsSignature}.
\end{proof}


\begin{appendix}

\section{The null-cobordism theorem for closed manifolds}\label{sec-closedmanifolds}

\noindent
In this appendix we discuss a version of the null-cobordism Theorem~\ref{CobordismInvarianceProductType} for closed manifolds. Most of the previous assumptions no longer explicitly appear in this version, e.g., we do not assume product type geometry, and there isn't an operator $A$ on $M$ at the outset, but symbolic assumptions instead. As mentioned in the introduction this is due to the richness of analytic tools available for this situation that allows to create the preconditions needed to apply Theorem~\ref{CobordismInvarianceProductType} instead of having to assume them from the outset.

Let $Y$ be a closed, compact Riemannian manifold and $E \to Y$ be a Hermitian vector bundle, and consider a family
\begin{equation}\label{Asigmafamily}
p(\sigma) = \sum\limits_{j=0}^{\mu}a_j(y,D_y)\sigma^j : C^{\infty}(Y;E) \to C^{\infty}(Y;E), \; \sigma \in \R,
\end{equation}
where $a_j(y,D_y) \in \Diff^{\mu-j}(Y;E)$, and $\mu \geq 1$. We assume that the parameter-dependent principal symbol
\begin{equation}\label{paramprincsymb}
\sym(p)(y,\eta;\sigma) = \sum\limits_{j=0}^{\mu}\sym(a_j)(y,\eta)\sigma^j : E_y \to E_y
\end{equation}
is invertible on $\bigl(T^*Y\times\R\bigr)\setminus 0$, and that $p(\sigma) = p(\sigma)^*$ is (formally) selfadjoint. By elliptic and analytic Fredholm theory,
$$
\R \ni \sigma \mapsto p(\sigma) : H^{\mu}(Y;E) \subset L^2(Y;E) \to L^2(Y;E)
$$
is a family of selfadjoint unbounded Fredholm operators acting in $L^2(Y;E)$ that is invertible for all $\sigma \in \R$ except at finitely many points, and it makes sense to consider the spectral flow
$$
\SF[p(\sigma)]:= \SF[p(\sigma) : H^{\mu}(Y;E) \subset L^2(Y;E) \to L^2(Y;E),\; -\infty < \sigma < \infty] \in \Z
$$
associated with $p(\sigma)$.

\begin{lemma}
The spectral flow is an invariant of the principal symbol \eqref{paramprincsymb} in the sense that if $p_j(\sigma)$, $j=1,2$, are two elliptic selfadjoint families of order $\mu \geq 1$ of the form \eqref{Asigmafamily} with $\sym(p_1)(y,\eta;\sigma) = \sym(p_2)(y,\eta;\sigma)$ then $\SF[p_1(\sigma)] = \SF[p_2(\sigma)]$.
\end{lemma}
\begin{proof}
Let $R > 0$ be such that
$$
p_1(\sigma) + s[p_2(\sigma)-p_1(\sigma)] : H^{\mu}(Y;E) \subset L^2(Y;E) \to L^2(Y;E)
$$
is invertible for $|\sigma| \geq R > 0$ and all $0 \leq s \leq 1$. Consequently, this family is a homotopy of selfadjoint Fredholm functions on $[-R,R]$, invertible at both endpoints, and by the homotopy invariance of the spectral flow for such families we see that $\SF[p_1(\sigma)] = \SF[p_2(\sigma)]$.
\end{proof}

Suppose there exists a compact Riemannian manifold $M$ with $\partial M = Y$. Utilizing the geodesic flow from the boundary in the direction of the inner normal vector field shows that there exists $\varepsilon > 0$ and a collar neighborhood map $U(Y) \cong [0,\varepsilon)\times Y$ near the boundary such that the metric in $U(Y)$ takes the form $dx^2 + g_Y(x)$ with a smooth family of metrics $g_Y(x)$ on $Y$, $0 \leq x < \varepsilon$, and such that $g_Y(0) = g_Y$ is the given metric on $Y$. Moreover, by choosing $\varepsilon >  0$ small enough, there exists a defining function for $\partial M$ on $M$ that in $U(Y)$ is represented by projection onto the coordinate in $[0,\varepsilon)$. We'll also denote this global defining function by $x : M \to \overline{\R}_+$. In particular,
$$
T^*M\big|_Y = T^*Y \oplus \Span\{dx\big|_Y\}
$$
subject to these choices, and we can split variables $(y,\eta;\sigma) \in T^*M\big|_Y$ accordingly.

\begin{theorem}[Null-Cobordism Theorem]\label{CobordismInv}
Let $M$ be a compact Riemannian manifold $M$ with $\partial M = Y$, and let $\E \to M$ be a Hermitian vector bundle with $\E\big|_Y = E$. Let $T^*M\bigl|_Y \cong T^*Y \times \R$ subject to the choices described above, and suppose there exists a symmetric, elliptic, differential principal symbol $a \in C^{\infty}(T^*M\setminus 0;\End(\pi^*\E))$ of order $\mu$ such that
$$
a(y,\eta;\sigma) = \sym(p)(y,\eta;\sigma) \textup{ for } (y,\eta;\sigma) \in \bigl(T^*M\setminus 0\bigr)\big|_Y,
$$
where $\pi : T^*M \to M$ is the canonical projection. Then $\SF[p(\sigma)] = 0$.
\end{theorem}

With the family $p(\sigma)$ from \eqref{Asigmafamily} we associate the indicial operator
\begin{equation}\label{Awedge}
A_{\wedge} = x^{-1}\sum\limits_{j=0}^{\mu}a_j(y,D_y)(xD_x)^j: C_c^{\infty}(\R_+\times Y;E) \subset L^2(\R_+\times Y;E) \to L^2(\R_+\times Y;E).
\end{equation}
Here we also write $E$ for its pull-back to $\R_+\times Y$ with respect to the projection onto $Y$, and equip $\R_+\times Y$ with the product metric $dx^2 + g_Y$. Then $A_{\wedge}$ is symmetric and densely defined. Let $\Dom_{\min}(A_{\wedge})$ be the domain of the closure, and $\Dom_{\max}(A_{\wedge})$ be the domain of the adjoint.

\begin{proof}[Proof of Theorem~\ref{CobordismInv}]
In the previously fixed collar neighborhood $U(Y) \cong [0,\varepsilon) \times Y$ we utilize standard deformations of the Riemannian metric on $M$, the Hermitian metric on $\E$, and the principal symbol $a$ to reduce to a product type structure near the boundary, as follows:

Pick an isomorphism $\E\big|_{U(Y)} \cong \pi_Y^*E$ that is the identity over $Y$, where $\pi_Y : [0,\varepsilon)\times Y \to Y$ is the projection map. With respect to the pull-back of the given Hermitian metric on $E$ to $\pi_Y^*E$, the metric on $\E\big|_{U(Y)}$ under this isomorphism is then represented by $h(x,y) \in C^{\infty}([0,\varepsilon)\times Y;\End(\pi_Y^*E))$ such that $h = h^* > 0$ and $h(0,y) = \textup{Id}$. Choose $C^{\infty}$-functions $\phi,\psi : [0,\varepsilon) \to \R$ with
\begin{gather*}
\phi \equiv 0 \textup{ on } 0 \leq x \leq \tfrac{\varepsilon}{3},\; 0 < \phi < \tfrac{2\varepsilon}{3} \textup{ on } \tfrac{\varepsilon}{3} < x < \tfrac{2\varepsilon}{3},  \textup{ and } \phi \equiv x \textup{ on } \tfrac{2\varepsilon}{3} \leq x < \varepsilon; \\
\psi  \equiv x \textup{ on } 0 \leq x \leq \tfrac{\varepsilon}{3},\; \psi > 0 \textup{ on } \tfrac{\varepsilon}{3} < x < \tfrac{2\varepsilon}{3}, \textup{ and } \psi \equiv 1 \textup{ on } \tfrac{2\varepsilon}{3} \leq x < \varepsilon.
\end{gather*}
We then deform the Riemannian metric on $U(Y)$ and Hermitian metric on $\E\big|_{U(Y)}$ to
$$
\tilde{g} = dx^2 + g_Y(\phi(x)) \textup{ and } \tilde{h}(x,y) = h(\phi(x),y) \in C^{\infty}([0,\varepsilon)\times Y;\End(\pi_Y^*E)),
$$
respectively, which both connect seamlessly with the Riemannian metric on $M$ outside $U(Y)$, and the Hermitian metric on $\E$. We also change the principal symbol in $\open U(Y)$ to
\begin{equation}\label{princsymbalt}
\tilde{a}(x,y,\eta;\sigma) = \psi(x)^{-1}a(\phi(x),y,\eta;\psi(x)\sigma) : E_y \to E_y
\end{equation}
for $(x,y,\eta;\sigma) \in T^*\bigl((0,\varepsilon)\times Y\bigr)\setminus 0$ with the obvious identifications of variables, which again connects seamlessly outside the collar neighborhood. The new homogeneous principal symbol $\tilde{a} \in C^{\infty}(T^*\open M\setminus 0,\End(\pi^*\E))$ is symmetric with respect to the new metric on $\E$, and elliptic over $\open M$. In $\open U(Y)$ we have
$$
\tilde{a}(x,y,\eta;\sigma) = x^{-1}\sym(p)(y,\eta;x\sigma) : E_y \to E_y \textup{ for } 0 < x < \tfrac{\varepsilon}{3}
$$
by construction, which aligns with the principal symbol of $A_{\wedge}$ from \eqref{Awedge}. Let now $A \in \Diff^{\mu}(\open M;\E)$ be symmetric $C_c^{\infty}(\open M;\E) \to C_c^{\infty}(\open M;\E)$ with respect to the $L^2$-inner product associated with the modified metrics on $M$ and $\E$, respectively, such that the principal symbol $\sym(A) = \tilde{a}$ on $T^*\open M \setminus 0$, and such that in $\open U(Y)$ we have $A = A_{\wedge}$ on $C_c^{\infty}((0,\frac{\varepsilon}{4})\times Y;E)$. Then
$$
A = x^{-1}P : C_c^{\infty}(\open M;\E) \subset L^2(M;\E) = x^{-\frac{1}{2}}L^2_b(M;\E) \to x^{-\frac{1}{2}}L^2_b(M;\E)
$$
is symmetric, and $P \in \Diff^{\mu}_b(M;\E)$ is $b$-elliptic (see \cite{RBM2}). Moreover, by construction $p(\sigma)$ is the indicial family of the operator $P$.

By analytic Fredholm theory $p(\sigma) : H^{\mu}(Y;E) \to L^2(Y;E)$ is invertible for $\sigma \in \C$ except for the discrete set $\spec_b(p)$. In the sequel it will be convenient to assume that $\spec_b(p)\cap\{\sigma \in \C;\; 0 < |\Im(\sigma)| \leq \frac{1}{2}\} = \emptyset$. As explained in Section~\ref{sec-indicialoperators}, this can be achieved by replacing $p(\sigma)$ by $p(t\sigma)$ for sufficiently small $t > 0$ if necessary, which does not impact the spectral flow. Moreover, the assumptions of the theorem pertaining to the principal symbol of $p(\sigma)$ also hold for $p(t\sigma)$; to see this pick a $C^{\infty}$-function $\chi : [0,\varepsilon) \to \R$ with
$$
\chi  \equiv t \textup{ on } 0 \leq x \leq \tfrac{\varepsilon}{3},\; \chi > 0 \textup{ on } \tfrac{\varepsilon}{3} < x < \tfrac{2\varepsilon}{3}, \textup{ and } \chi \equiv 1 \textup{ on } \tfrac{2\varepsilon}{3} \leq x < \varepsilon,
$$
and alter the principal symbol \eqref{princsymbalt} in $\open U(Y)$ to
$$
\tilde{a}(x,y,\eta;\sigma) = \psi(x)^{-1}a(\phi(x),y,\eta;\psi(x)\chi(x)\sigma) : E_y \to E_y
$$
for $(x,y,\eta;\sigma) \in T^*\bigl((0,\varepsilon)\times Y\bigr)\setminus 0$. We may thus proceed without loss of generality under the assumption that $\spec_b(p)\cap\{\sigma \in \C;\; 0 < |\Im(\sigma)| \leq \frac{1}{2}\} = \emptyset$. In view of Section~\ref{sec-indicialoperators} for $A_{\wedge}$ and by invoking elliptic regularity for $A$ we then get
\begin{align*}
\Dom_{\min}(A_{\wedge}) &= x^{\frac{1}{2}}H^{\mu}_b(\R_+;L^2(Y;E))\cap x^{\frac{1}{2}}L^2_b(\R_+;H^{\mu}(Y;E))\cap L^2(\R_+\times Y;E), \\
\Dom_{\min}(A) &= x^{\frac{1}{2}}H^{\mu}_b(M;\E), \\
\intertext{and}
\Dom_{\max}(A_{\wedge}) &= \Dom_{\min}(A_{\wedge}) \oplus \bigoplus\limits_{\sigma_0 \in \spec_b(p)\cap\R}\E_{\sigma_0}(p), \\
\Dom_{\max}(A) &= \Dom_{\min}(A) \oplus \bigoplus\limits_{\sigma_0 \in \spec_b(p)\cap\R}\E_{\sigma_0}(p),
\end{align*}
where $\E_{\sigma_0}(p)$ is defined as in \eqref{Esigma0pdef} based on a cut-off function $\omega \in C_c^{\infty}([0,\frac{\varepsilon}{4}))$ with $\omega \equiv 1$ near $x = 0$ so that elements in $\E_{\sigma_0}(p)$ can interchangeably be regarded both as sections of $E$ on $\R_+\times Y$, as well as sections of $\E$ on $M$ supported near the boundary. In particular, this implies that
$$
\bigl(\Dom_{\max}(A)/\Dom_{\min}(A),[\cdot,\cdot]_A\bigr) \cong \bigl(\Dom_{\max}(A_{\wedge})/\Dom_{\min}(A_{\wedge}),[\cdot,\cdot]_{A_{\wedge}}\bigr)
$$
because $[u,v]_{A_{\wedge}} = [u,v]_{A}$ for $u,v \in \bigoplus\limits_{\sigma_0 \in \spec_b(p)\cap\R}\E_{\sigma_0}(p)$ by construction. Finally, it remains to note that $\Dom_{\max} \hookrightarrow x^{-\frac{1}{4}}H^{\mu}_b(M;\E)$, and the embedding $x^{-\frac{1}{4}}H^{\mu}_b(M;\E) \hookrightarrow x^{-\frac{1}{2}}L^2_b(M;\E) = L^2(M;\E)$ is compact.
\end{proof}

Theorem~\ref{CobordismInv} and Lemma~\ref{SpectralFlowIndex} imply:

\begin{corollary}[Cobordism Invariance of the Index]
Suppose that $E = E_- \oplus E_+$ is an orthogonal direct sum, and that the family \eqref{Asigmafamily} is of the form
$$
{\mathscr D}(\sigma) = \begin{bmatrix}
\sigma & D^* \\ D & -\sigma
\end{bmatrix} : C^{\infty}\Biggl(Y;\begin{array}{c} E_- \\ \oplus \\ E_+\end{array}\Biggr) \to C^{\infty}\Biggl(Y;\begin{array}{c} E_- \\ \oplus \\ E_+\end{array}\Biggr), \; \sigma \in \R,
$$
where $D : C^{\infty}(Y;E_-) \to C^{\infty}(Y;E_+)$ is an elliptic differential operator of first order, and $D^* : C^{\infty}(Y;E_+) \to C^{\infty}(Y;E_-)$ is its (formal) adjoint. Then
$$
\SF[{\mathscr D}(\sigma)] = \ind D = \dim\ker(D) - \dim\ker(D^*).
$$
In particular, if the assumptions of Theorem~\ref{CobordismInv} hold, then $\ind(D) = 0$.
\end{corollary}

\end{appendix}


\end{document}